\documentclass[11pt]{article}

\usepackage{amsmath,amssymb,latexsym,amsthm,graphicx}

\usepackage{mathrsfs}

\newtheorem{theorem}{Theorem}[section]

\newtheorem{lemma}[theorem]{Lemma}

\newtheorem{prop}[theorem]{Proposition}

\newtheorem{rem}[theorem]{Remark}

\newcommand{\supp}{\operatorname{supp}}

\newcommand{\mS}{\mathcal{S}}

\newcommand{\mA}{\mathcal{A}}

\newcommand{\mC}{\mathcal{C}}

\newcommand{\mR}{\mathcal{R}}

\newcommand{\mP}{\mathcal{P}}

\newcommand{\mD}{\mathcal{D}}

\newcommand{\Int}{\mbox{Int}}

\newcommand{\rN}{\mathbb{R}}

\newcommand{\mT}{\mathcal{T}}

\newcommand{\intl}{\int\limits}

\newcommand{\ma}{{\bf a}}

\newcommand{\mb}{{\bf b}}

\newcommand{\cT}{\mathbb{T}}

\newcommand{\mV}{\mathcal{V}}

\newcommand{\wf}{\mbox{WF}}

\newcommand{\eg}{\varepsilon}

\newcommand{\llg}{\lambda}

\newcommand{\Llg}{\Lambda}

\newcommand{\ve}{{\bf e}}

\newcommand{\vv}{{\bf v}}

\newcommand{\sg}{\sigma}

\newcommand{\Og}{\Omega}

\newcommand{\Ga}{\Gamma}

\newcommand{\pdh}{\partial}

\title{On Artifacts in Limited Data Spherical Radon Transform: Flat Observation Surfaces}

\author{Linh V. Nguyen}

\begin{document}

\maketitle

\begin{abstract} In this article, we characterize the strength of the reconstructed singularities and artifacts in a reconstruction formula for limited data spherical Radon transform. Namely, we assume that the data is only available on a closed subset $\Gamma$ of a hyperplane in $\mathbb{R}^n$ ($n=2,3$). We consider a reconstruction formula studied in some previous works, under the assumption that the data is only smoothened out to a finite order $k$ near the boundary. For the problem in the two dimensional space and $\Gamma$ is a line segment, the artifacts are generated by rotating a boundary singularity along a circle centered at an end point of $\Gamma$. We show that the artifacts are $k$ orders smoother than the original singularity. For the problem in the three dimensional space and $\Gamma$ is a rectangle, we describe that the artifacts are generated by rotating a boundary singularity around either a vertex or an edge of $\Gamma$. The artifacts obtained by a rotation around a vertex are $2k$ orders smoother than the original singularity. Meanwhile, the artifacts obtained by a rotation around an edge are $k$ orders smoother than the original singularity. For both two and three dimensional problems, the visible singularities are reconstructed with the correct order. We, therefore, successfully quantify the geometric results obtained recently by J. Frikel and T. Quinto.  
\end{abstract}

\section{Introduction}\label{S:intro} Let $\mS \subset \rN^n$ be the hyperplane $\mS=\{z= (0,z'): z' \in \rN^{n-1}\}.$
We consider the following (restricted) spherical Radon transform $\mR f$ of a function $f$ defined in $\rN^n$
$$\mR f (z, t) = \int\limits_{S(z,t)} f(y) \, d\sg(y),\quad (z,t) \in \mS \times (0,\infty).$$
Here, $S(z,t)$ is the sphere centered at $z$ of radius $t$, and $d\sg$ is its surface measure. This transform appears in several imaging modalities, such as \footnote{The reference list is highly incomplete. The interested reader is suggested to explore the literature for the comprehensive references to each imaging modality.} thermo/photoacoustic tomography \cite{FPR,FHR,KKun}, ultrasound imaging \cite{norton79,norton81}, SONAR \cite{QuintoSONAR}, SAR \cite{cheney00tomography,NoChe,SUSAR} and inverse elasticity  \cite{BuKar}. 

\medskip

In many applications, it is assumed that $f$ is compactly supported inside the half space $$\Og:=\rN^n_+=\{x: x_1>0\}.$$
Then, $f$ can be reconstructed from $g=\mR(f)$ by the filtered back projection formula (see, e.g., \cite{BuKar,NaRa,Beltukov})
$$f(x) = \frac{x_1}{\pi^n} (\mR^* \mP \mR f)(x).$$
Here, $\mR^*$ is the formal $L^2$-adjoint of $\mR$
$$\mR^*(g)(x)= \intl_{\mS} g(z,|x-z|) d\sg(z),$$
and $\mP$ is the pseudo-differential operator of order $(n-1)$ defined by
$$\mP(h)(r) = \intl_\rN \, \intl_{\rN_+} e^{i(\tau^2-r^2)\llg} \, |\llg|^{n-1} \, h(\tau) \, d\tau \, d\llg.$$

\medskip

However, in real applications, the knowledge of $\mR f$ is only available on a closed bounded subset $\Gamma \subset \mS$ with nontrivial interior (see, e.g., \cite{XWAK08,QuintoSONAR}).  The following formula is proposed for the approximate construction in several works (see, e.g., \cite{AMP,XWAK08})
$$ \mT f (x) := \frac{x_1}{\pi^n} (\, \mR^* \, \mP \chi \, \mR f)(x).$$
Here, $$\mbox{\bf $\chi \in C^\infty(\Ga)$ and $\chi \equiv 0$ on $\mS \setminus \Ga$ is the data cut-off function}.$$ 

\medskip

It has been commonly assumed that $\chi \in C^\infty(\mS)$. We prove in \cite{AMP} that, under this assumption, $\mT$ is a pseudo-differential operator with the principal symbol
$$\sg_0(x,\xi) = \chi(z),$$ where $z$ is the intersection of the line $\ell(x,\xi)$ through $x$ along direction of $\xi$ with the plane $\mS$. The assumption that $\chi$ is infinitely smooth is essential to apply the theory of pseudo-differential operator to analyze $\mT$. In particular, it implies $\wf(\mT f) \subset \wf(f)$. The multiplication with such function $\chi$ is considered as an infinite smoothening.  However, it is known to eliminate some singularities (image features) pointing near the boundary of $\Ga$ (see more discussion in \cite{FQ14} and reference therein). Therefore, one might consider to use other kind of smoothening (e.g., of finite order \footnote{We say that $\chi$ is smoothening of order $k$ if $\chi \in C^{k-1}$ across $\pdh \Ga$.}) or no smoothening at all \footnote{That is, $\chi \equiv 1$ on $\Ga$.}. However, these choices are shown to generate the artifacts into the pictures (e.g., \cite{halt08re,buehler2011model,patrickeyev2004removing}). In order for the practitioners to make the correct choice, it is important to analyze the effect of $\mT$ under these situations. The recent work by J. Frikel and T. Quinto \cite{FQ14} provides a nice geometric description for what happens\footnote{Their setting is a little bit different from ours. However, their results translate without major modifications.}. In particular, they show that the visible singularities are reconstructed and the artifacts would occur in some specific pattern (see more discussion below). Our main goal is to characterize the strength of the reconstructed singularities and, more importantly, the artifacts.

\medskip

It is worth mentioning that similar problem has been studied for the X-ray (or classical Radon) transform \cite{RamKat-AML1,RamKat-AML2,Kat-JMAA,FQ13,Streak-Artifacts}. Although this article shares the same spirit with our previous work \cite{Streak-Artifacts}, the technique developed here is original and different from there. Namely, in this article, for each type of singularities, we have to make a correct choice of the oscillatory representation of $\mT$ to work with. Moreover, the microlocal analytic arguments employed here are more sophisticated than there. 

\medskip

A different but related topic is to analyze the imaging scenarios when the associated canonical relation is not a local graph (see, e.g., \cite{MelroseTaylor,GrUDuke,Nolan-Symes,Nolan}, just to name a few). The pioneering work by Greenleaf and Uhlmann \cite{GrUDuke,GrUCon} employs the theory of class $I^{p,\ell}$ (see \cite{MU,GU,AnUhl}) to analyze such a situation appearing in the X-ray transform. The same technique has been exploited successfully in other situations (e.g., \cite{FLU,NoChe,FeleaCPDE,FeleaQuinto,Suresh,FeleaSAR,Am-singular}). Although a direct use of such an idea does not seem to work in our situation, our approach is inspired by the same spirit. Namely, we analyze the microlocal behavior of the associated FIO (or more precisely its Schwartz kernel) on its intersecting Lagrangians. 

\medskip

For the case $n=2$, using the argument as in \cite{FQ14}, we show that the artifacts are generated by rotating the boundary singularities \footnote{The definition of boundary singularities is in Section \ref{S:2dim}.} along circles whose center is an end point of $\Ga$ (see Proposition \ref{P:wavefront} and Discussion~1). Moreover, employing some asymptotic arguments,  we analyze the strength of these artifacts. Namely, assume that $\chi$ vanishes to order $k$ at the end points of $\Ga$ \footnote{That is, $\chi$ is smoothening of order $(k-1)$}. We prove that the artifacts are $k$ orders smoother than the original singularities  (see Theorem \ref{T:Main1} and Discussion~2).  We go further to analyze the problem in three dimensional space. We consider $\Ga$ to be the rectangle $\{0\} \times [-a,a] \times [-b,b]$ and $\chi(z) = h_2(z_2) h_3(z_3)$ vanishes on the edges of $\Ga$ to order $k$. We, using the arguments as in \cite{FQ14}, show that the singularities propagate by rotating around either a vertex or an edge of $\Ga$ (see Proposition \ref{P:wavefront-3} and Discussion~3). Moreover, employing delicate asymptotic and microlocal analytic arguments, we prove that the artifacts generated by a rotation around a vertex are $2k$ orders smoother than the original singularity, while those from a rotation around an edge are $k$ orders smoother than the original singularity (See Theorem \ref{T:Main2} and Discussion~4). Finally, we mention that, for both two and three dimensional problems, all the visible singularities \footnote{Definition of visible singularities is provided in Sections~\ref{S:2dim} and \ref{S:3dim}.} are reconstructed with the correct order. This follows from a similar argument for full data problem in \cite{AMP} (see Theorems \ref{T:Main1}~a) and \ref{T:Main2}~a)). 

\medskip

The article is organized as follows. In Section \ref{S:2dim} we consider the two dimensional problem. The three dimensional problem is studied in Section \ref{S:3dim}. Some concluding remarks are provided in Section~\ref{S:Re}. Finally, some asymptotics results  and background knowledge in microlocal analysis are provided in Appendix. 


\section{Two dimensional problem} \label{S:2dim}
Let us consider the problem in the two dimensional space $\rN^2$. Without loss of generalities, we assume that $\Gamma=\{(0,z_2): -1\leq z_2\leq 1\}$. 
As mentioned in the introduction, we will analyze $\mT$ when $\chi$ is not infinitely smooth at the points $z_{\pm}=(0,\pm 1)$. For the notational simplicity, we will assume that $\chi$ vanishes to the same order $k$ at $z_\pm$, although our analysis works equally well for the case $\chi$ vanishes to two different orders at these two points. 

\medskip

Let us define $\mV \subset \cT^* \Og \setminus 0$ by $$\mV= \{(x,\xi) \in \cT^* \Og \setminus 0: \ell(x,\xi) \mbox{ intersects } \Ga \}.$$
Its boundary $\pdh \mV$ and interior $\Int(\mV)$ are defined by 
$$\pdh \mV= \{(x,\xi) \in \cT^* \Og \setminus 0: \ell(x,\xi) \mbox{ intersects } \pdh(\Ga)\},$$
and 
$$\Int (\mV) = \{(x,\xi) \in \cT^* \Og \setminus 0: \ell(x,\xi) \mbox{ intersects } \Int(\Ga) \}.$$

\medskip

In the literature of spherical Radon transform, $\Int(\mV)$ is called the {\bf audible} or {\bf visible zone} since any singularity $(x,\xi) \in \wf(f) \bigcap \Int(\mV)$ creates a corresponding singularity in the limited data $g = \mR f|_{\Ga \times \rN_+}$ (see, e.g., \cite{louis00local}). We also call $\pdh \mV$ and $\cT^* \Og \setminus \mV$ {\bf boundary} and {\bf invisible} zones respectively. A singularity $(x,\xi) \in \wf(f)$ will be called visible, boundary, or invisible accordingly to the zone it belongs to.

\medskip
 
\noindent We define the following canonical relations in $(\cT^* \Og \setminus 0) \times (\cT^* \Og \setminus 0)$
$$\Delta_\mV = \{(x,\xi; x, \xi): (x,\xi) \in \mV\},$$
and
$$\Llg_\pm = \{\big(x,\tau(x-z_\pm);\, y,\tau(y-z_\pm) \big) \in (\cT^* \Og \setminus 0) \times (\cT^* \Og \setminus 0): |x-z_\pm| = |y-z_\pm| \}.$$
We notice that $(x,\xi;\, y, \eta) \in \Llg_\pm$ iff $(y,\eta)$ is in the boundary zone and $(x,\xi)$  is obtained from $(y,\eta)$ by a rotation around the corresponding boundary point $z_\pm$ (see Figure \ref{fig:zpm}). 

\begin{figure}
\begin{center}
    \includegraphics[width=0.2\textwidth]{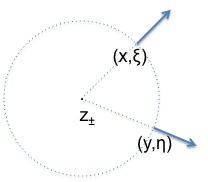}
 \end{center}
    \caption{Rotation around $z_\pm$}
    \label{fig:zpm}
\end{figure}

\medskip

\noindent The following result characterizes the wave front set of the Schwartz kernel $\mu$ of $\mT$ \footnote{We recall the twisted canonical $\mC'$ associated to $\mC$ is defined by $\mC'=\{(x,\xi; y, \eta): (x,\xi; y, -\eta) \in \mC\}$.} 
\begin{prop} \label{P:wavefront}
We have $$\wf(\mu)' \subset \Delta_\mV \cup \Llg_+ \cup \Llg_-.$$
\end{prop}
The proof of Proposition \ref{P:wavefront} follows from the standard calculus of wave front set (see, e.g., \cite{Ho1} and Appendix~\ref{A:Cal}). It was first presented in \cite{FQ14} for the case $\Ga$ is half a circle. The argument applies to our situation without any major modifications. We present it here for the sake of completeness (and to motivate further discussion).
\begin{proof}[\bf Proof of Proposition~\ref{P:wavefront}]
Let us denote by $\mu_R$ the Schwartz kernel of $\mR$. We notice that $\mR$ is an FIO with the phase function (see, e.g., \cite{louis00local,Pal-IPI,AMP})
$$\phi(x,z,r,\llg) = (|x-z|^2 - r^2) \llg.$$ For simplicity, we will identify $\mS$ with $\rN$ by the mapping
$$z=(0,z_2) \to z_2.$$
We have, considering $\mu_R$ as a function of $(z_2,r,x)$, 
$$\wf(\mu_R) \subset \mC_R: = \{\big((z_2,r), (\tau \, (z_2-x_2), - \tau \, r) ; \, x,\tau \, (x-z) \big): (z_2,r,x) \in \rN \times \rN_+ \times \Og:  |x-z|=r,\, \tau \neq 0\}.$$

\medskip

\noindent Also considering $\chi(z)$ as a function of $(z_2,r,x)$, we have
$$\wf(\chi) \subset \mC_\chi:= \{\big((z_2,r), \, (\tau', 0); \, x,{\bf 0}\big): z_2 = \pm 1, \, \tau' \neq 0\}.$$

\medskip

\noindent Applying the product rule for wave front set (see Appendix~\ref{A:Cal}), we obtain
$$\wf(\chi \mu_R) \subset \mC_R \cup  \mC_A ,$$
where
$$\mC_A = \{\big((z_2,r), \, (\tau \, (z_2-x_2)+ \tau', - \tau \, r); \, x,\tau \, (x-z)): z_2 = \pm 1, \, \tau' \neq 0\}.$$
Let $\mu^*_R$ be the Schwartz kernel for $\mR^*$. We have
$$\wf(\mu_R^*) \subset \mC_R^t.$$
We notice that  \footnote{$\mC^t$ is the transpose relation of $\mC$: $\mC^t=\{(y,\eta; x,\xi): (x,\xi; y, \eta) \in \mC\}.$}
\begin{eqnarray*} \mC_R^t \circ \mC_R &=& \Delta_\mV,\\[6 pt]
\mC_R^t \circ \mC_A &=&  \Llg_+ \cup \Llg_-.\end{eqnarray*}
Due to the composition rule for wave front set (see Appendix~\ref{A:Cal}), we conclude that \footnote{Since $\mP$ is a pseudo-differential operator, it does not change the wave front set of a function.}
$$\wf(\mu)' \subset \mC_R^t \circ (\mC_R \cup \mC_A)  \subset \Delta_\mV \cup \Llg_+ \cup \Llg_-.$$
\end{proof}

\medskip
\noindent {\bf Discussion 1.} Let us denote by $\pi_R$ the right projection of the product space $\cT^* \Og \times \cT^* \Og$ \footnote{That is $\pi_R(x,\xi; y, \eta) = (y,\eta)$.}. We have
$$\pi_R(\Delta_\mV)= \mV, \quad  \pi_R(\Llg_+ \cup \Llg_-)= \pdh \mV.$$
We obtain from Proposition \ref{P:wavefront}
$$\pi_R [\wf(\mu)']  \subset \mV.$$
We also recall the following rule for wave front sets (see Appendix~\ref{A:Cal}):
$$\wf(\mT f) \subset \wf(\mu)' \circ \wf(f).$$
The following effects of $\mT$ on the wave front set of $f$ can be deduced from Proposition \ref{P:wavefront} \footnote{The reader should be aware that $(x,\xi)$ in the below discussion  may play the role of $(y,\eta)$ in the definition of $\Llg_\pm$.}:
\begin{itemize}
\item[a)] If $(x,\xi) \in \wf(f)$ is an invisible singularity, then $(x,\xi) \not \in \pi_R(\wf(\mu)')$. Therefore, $(x,\xi)$ is completely smoothened out by $\mT$ (i.e., it is not reconstructed and it does not generate any artifact).

\item[b)] If $(x,\xi) \in \wf(f)$ is a visible singularity, then $(x,\xi) \in \pi_R(\Delta_\mV)$ and $(x,\xi) \not \in \pi_R(\Llg_+ \bigcup \Llg_-)$. Therefore, $(x,\xi)$ may be reconstructed but it does not generate any artifacts.

\item[c)] If $(x,\xi) \in \wf(f)$ is a boundary singularity, then $(x,\xi) \in \pi_R(\Llg_+)$ or $(x,\xi) \in \pi_R(\Llg_-)$. Therefore, $(x,\xi)$ may generate artifacts by rotating around $z_+$ or $z_-$, respectively. 

\end{itemize}

The above observation is similar to \cite[Remark 4.1]{FQ14}. They provide geometric descriptions for the reconstruction of original singularities and the generation of artifacts. We now proceed to obtain more quantitative results. Let us first recall our assumption on $\chi$: it vanishes to order $k$ at $z_\pm=(0,\pm 1)$. That is, we can write
\begin{equation} \label{E:chi2} \chi(0,\tau) = h(\tau),\quad \tau \in [-1,1],\end{equation} where $h \in C^\infty([-1,1])$ vanishes to order $k$ at $\tau=\pm1$ \footnote{That is $h^{(l)}(\pm 1) =0$ for all $0 \leq l \leq k-1$ and $h^{(k)}(\pm 1) \neq 0$.}. Here is our main result of this section:
\begin{theorem} \label{T:Main1} The following statements hold:
\begin{itemize}
\item[a)] Microlocally on $\Delta \setminus (\Llg_+ \cup \Llg_-)$, we have $\mu \in I^0(\Delta)$ with the principal symbol
$$\sg_0(x,\xi) = \chi(z),$$ where $z$ is the intersection of the line $\ell(x,\xi)$ with $\mS$.
\item[b)] Microlocally on $\Llg_\pm \setminus \Delta$, $\mu \in I^{-k-\frac{1}{2}}(\Llg_\pm).$ 
\end{itemize}
\end{theorem}
\noindent The reader is referred to Appendix~\ref{A:FIO} for the class $I^m(\Llg_\pm)$ (see the discussion on the Fourier distributions whose canonical relation is defined by rotations around a point). We would like to mention that, as shown in the proof below, the order of $\mu$ stated in Theorem~\ref{T:Main1}~b) is optimal. We now proceed to prove Theorem~\ref{T:Main1}. 

\begin{proof}[\bf Proof of Theorem \ref{T:Main1}]  Let us first prove b). To this end, we write $\mu$ in the following form (see \cite{AMP})
$$\mu(x,y) = \frac{x_1}{\pi^2}\intl_{\rN} \intl_{\Ga}e^{i(|x-z|^2-|y-z|^2) \, \llg}\,|\llg| \, \chi(z) \, dz \, d\,\llg.$$
The phase function of $\mu$ can be written as
$$(|x-z|^2 -|y-z|^2) \llg = \left<x-y,x+y-2z \right> \llg.$$
Therefore,
\begin{eqnarray*} \mu(x,y) &=&  \frac{x_1}{\pi^2} \intl_{\rN} \intl_{\Ga}e^{i\left<x-y,x+y-2z\right> \llg}\,|\llg| \,\chi(z) \, dz \, d\,\llg\\
&=&  \frac{x_1}{\pi^2} \intl_{\rN} e^{i\left<x-y,x+y \right> \llg}\,|\llg| \, \intl_{\Ga}e^{- i\left<x-y,2z\right> \llg}\, \chi(z) \, dz \, d\,\llg. \end{eqnarray*}
We obtain, using the formula (\ref{E:chi2}),
\begin{eqnarray} \label{E:mu-21} \mu(x,y) &=&  \frac{x_1}{\pi^2}  \intl_{\rN} e^{i\left<x-y,x+y \right> \llg}\,|\llg| \, \intl_{-1}^1 e^{- 2 i(x_2-y_2) \, \tau \llg}\, h(\tau) \, d\tau \, d\,\llg
.\end{eqnarray}
Let us decompose $h$ into two parts $$h(\tau) = h_+(\tau) + h_-(\tau),\quad \tau \in [-1,1].$$ Here, $h_\pm \in C^\infty[-1,1]$ such that $h_\pm(\tau)$ vanishes when $\tau$ is close to $\mp 1$ (i.e., $h_\pm(\tau)=h(\tau)$ near $\tau = \pm 1$). We then write
$$\mu = \mu_+ + \mu_-,$$
where
\begin{eqnarray*} \mu_\pm(x,y)=  \frac{x_1}{\pi^2}  \intl_{\rN} e^{i\left<x-y,x+y \right> \llg}\,|\llg| \, \intl_{-1}^1 e^{- 2 i(x_2-y_2) \, \tau \llg}\, h_\pm(\tau) \, d\tau \, d\,\llg
.\end{eqnarray*}
An argument similar to the proof of Proposition \ref{P:wavefront} shows that $$\wf(\mu_\pm)' \subset \Delta \cup \Llg_\pm.$$
It now suffices to show that microlocally on $\Llg_\pm \setminus \Delta$, $\mu_\pm \in I^{-k-\frac{1}{2}}(\Llg_\pm)$. 

\medskip

Let us consider $\mu_+$. Due to Lemma \ref{L:sym} (see Appendix~\ref{A:Asym}):
\begin{eqnarray*}  \intl_{-1}^1 e^{- 2 i(x_2-y_2) \, \tau \llg}\, h_+(\tau) \, d\tau \, d\,\llg = - e^{ - 2 \, i \, (x_2-y_2) \, \llg} \,a(x,y,\llg)
\end{eqnarray*} where $a(x,y,\llg)$ is a classical symbol of order $-k-1$ when $x_2 \neq y_2$, with the leading term
\begin{equation} \label{E:ak} a_{-k-1}(x,y,\llg) = \frac{- h^{(k)}(1)}{[2\, i \,(x_2-y_2) \llg]^{(k +1)}}.\end{equation}
We arrive to,
\begin{eqnarray*} \mu_+(x,y)  &=&  \frac{x_1}{\pi^2}  \intl_{\rN} e^{i \big[\left<x-y,x+y \right>- 2 (x_2-y_2) \big] \llg} \,|\llg| \,a(x,y,\llg) \, d\,\llg
\\&=&  \frac{x_1}{\pi^2}  \intl_{\rN} e^{i (|x-z_+|^2 - |y-z_+|^2) \llg} \,|\llg| \, a(x,y,\llg) \, d\,\llg.\end{eqnarray*}

Let $(x^*,\xi; y^*, \eta) \in \Llg_+ \setminus \Delta$. In a small neighborhood of $(x^*,y^*) \in \Og \times \Og$, we have $x_2 \neq y_2$. Therefore, the amplitude function in the above integral is a symbol of order $-k$ near $(x^*,y^*)$. Therefore, microlocally near $(x^*,\xi; y^*, \eta)$, $\mu_+ \in I^{-k - \frac{1}{2}}(\Llg_+)$ (see Appendix~\ref{A:FIO} for the discussion on the Fourier distributions whose canonical relation is defined by rotations around a point). 

\medskip

This finishes the proof for $\mu_+$. The proof for $\mu_-$ is similar. We, hence, have finished the proof of b). We notice that the leading term $a_{-k-1}(x,y,\llg)$, given by (\ref{E:ak}), of $a(x,y,\llg)$ is nonzero. Therefore, the order $-k-\frac{1}{2}$ of $\mu_+$ on $\Llg_+$, obtained above, is optimal. That is, the order $-k-\frac{1}{2}$ of $\mu$ on $\Llg_\pm$ is also optimal.

\medskip

It now remains to prove a). We recall the following formula of $\mu$ derived in \cite{AMP}
\begin{eqnarray}\label{E:mu-22}
\mu(x,y) = \frac{1}{(2 \pi)^2} \intl_{\rN^2} e^{i \big(\left<x-y, \xi \right> - \, |x-y|^2 \frac{\xi_1}{2 \, x_1} \big)} \chi(z) \, d \xi. 
\end{eqnarray}

\noindent We note that $\sg_0(x,\xi) :=\chi(z)$ is smooth microlocally near any $(x,\xi)$ such that $(x,\xi;x,\xi) \in \Delta \setminus (\Llg_+ \cup \Llg_-)$. Moreover, it is homogenous of degree $0$ in $\xi$. Therefore, near $\Delta \setminus (\Llg_+ \cup \Llg_-)$, $\mu \in I^{0}(\Delta)$. Finally, the principal symbol of $\mu$ is indeed $\sg_0(x,\xi)$, due to \cite[Theorem 3.2.1]{Soggeb}. 
\end{proof}

\medskip

\noindent {\bf Discussion 2.} The following improvement of b) and c) in Discussion 1 are clear consequences of Theorem \ref{T:Main1} \footnote{As in Discussion 1, $(x,\xi)$ in below discussion may play the role of $(y,\eta)$ in the definition of $\Llg_\pm$.}: 
\begin{itemize}
\item[b')] Let $(x,\xi) \in \wf(f)$ be a visible singularity. Due to the Theorem~\ref{T:Main1}~a), $\mT f$ reconstructs $(x,\xi)$ as long as $\chi(z) \neq 0$. Moreover, the reconstructed singularity is of the same order as the original singularity.
\item[c')] Let $(x,\xi) \in \wf(f)$ be a boundary singularity. Due to Theorem~\ref{T:Main1}~b) and Lemma~\ref{L:x0}, the artifacts generated by $(x,\xi)$ are $k$ orders smoother than the original singularity. \footnote{For certain kind of conormal singularities, the artifacts are $k+\frac{1}{2}$ orders smoother than the original singularity. We will discuss this in details in our upcoming publication \cite{Artifact-Sphere-Curve}.}
\end{itemize}
Theorem \ref{T:Main1} provides a quantitative improvement for Proposition \ref{P:wavefront}. However, it still does not answer the following question: whether a boundary singularity is reconstructed and, if yes, how strong is the reconstruction? We plan to address this issue in the future.



\section{Three dimensional problem} \label{S:3dim}
Let us now consider $n=3$. We assume that $$\Ga =\{0\} \times [-a,a] \times [-b,b] =\{(0,z_2,z_3):(z_2,z_3) \in [-a,a] \times [-b,b] \},$$ where $a,b>0$. We will analyze $\mT$ when $\chi$ is not infinitely smooth at the boundary of $\Ga$. For the sake of simplicity, we will assume that
\begin{equation} \label{E:chi} \chi(z) = h_2(z_2) \, h_2(z_3), \quad \mbox{ for all } z= (0,z_2,z_3),\end{equation}
where $h_2$ and $h_3$ are smooth on $[-a,a]$ and $[-b,b]$ and they both vanish to order $k$ at $\pm a$ and $\pm b$ respectively. 

\medskip

\noindent Similarly to the case $n=2$, we define
$$\mV= \{(x,\xi) \in \cT^* \Og \setminus 0: \ell(x,\xi) \mbox{ intersects } \Ga \},$$
and
$$\Delta_\mV = \{(x,\xi; x, \xi): (x,\xi) \in \mV\}.$$
We also define the visible, boundary, and invisible zones by $\Int(\mV)$, $\pdh \mV$, and $\cT^* \Og \setminus \mV$ respectively. 
Let us denote by $\vv_j$, $j=1,\dots, 4$, the vertices of $\Ga$
$$\vv_1=(0,a,b), \quad \vv_2 = (0,a,-b), \quad \vv_3=(0,-a,b), \quad \vv_4 = (0,-a,-b),$$
and $\ve_j$, $j=5,\dots,8$, the edges of $\Ga$
\begin{eqnarray*}
\ve_5 &=& \{0\} \times[-a, a] \times \{b\} = \{z: z=(0,-a \leq z_2\leq a,b)\}, \\[6 pt]
\ve_6 &=&\{0\} \times [-a, a] \times \{-b\} =\{z:z=(0,-a\leq z_2 \leq a,-b)\} ,\\[6 pt]
\ve_7 &=& \{0\} \times\{-a\} \times [-b,b] = \{z: z=(0,-a,-b\leq z_3\leq b)\},\\[6 pt]
\ve_8 &=& \{0\} \times \{a\} \times [-b,b]=\{z: z=(0,a, -b\leq z_3 \leq b)\}.
\end{eqnarray*}
We will denote 
$$\mV_{\vv_j} = \{(x,\xi) \in \cT^* \Og \setminus 0: \ell(x,\xi) \mbox{ passes through the vertex }  \vv_j\}.$$
and 
$$\mV_{\ve_j} = \{(x,\xi) \in \cT^*\Og \setminus 0: \ell(x,\xi) \mbox{ intersects the edge $\ve_j$ of }  \Ga \mbox{ not at a vertex} \}.$$
We note that
$$\pdh \mV = \big(\bigcup \mV_{\vv_j} \big) \bigcup \big(\bigcup \mV_{\ve_j} \big) := \mV_c \bigcup \mV_e.$$
We will call $\mV_c$ and $\mV_e$ {\bf corner} and {\bf edge} zones respectively. Also, $(x,\xi) \in \wf(f)$ is called a corner or edge singularity accordingly to the zone it belongs to. A boundary singularity, hence, is either a corner or edge singularity.


\medskip

We denote by $\Llg_j$, $j=1,\dots,4$, the following canonical relation in $(\cT^* \Og \setminus 0) \times (\cT^* \Og \setminus 0)$
$$\Llg_j = \{(x,\tau(x-\vv_j); \, y,\tau(y-\vv_j)) \in (\cT^*\Og \setminus 0) \times (\cT^*\Og \setminus 0): |x-\vv_j| = |y-\vv_j| \},$$
and 
\begin{multline*}
\Llg_5 = \{(x, \tau \, (x-z); \, y, \tau \, (y-z)) \in (\cT^*\Og \setminus 0) \times (\cT^*\Og \setminus 0): z \in \ve_5,~ x_2=y_2,~ |x-z| = |y-z|\},
\end{multline*}
\begin{multline*}
\Llg_6 = \{(x, \tau \, (x-z); \, y, \tau \, (y-z)) \in (\cT^*\Og \setminus 0) \times (\cT^*\Og \setminus 0): z \in \ve_6,~ x_2=y_2,~ |x-z| = |y-z|\},
\end{multline*}
\begin{multline*}
\Llg_7 = \{(x, \tau \, (x-z);\, y, \tau \, (y-z)) \in (\cT^*\Og \setminus 0) \times (\cT^*\Og \setminus 0):z \in \ve_7,~ x_3=y_3,~ |x-z| = |y-z|\},
\end{multline*}
\begin{multline*}
\Llg_8 = \{(x, \tau \, (x-z); \, y, \tau \, (y-z)) \in (\cT^*\Og \setminus 0) \times (\cT^*\Og \setminus 0): z \in \ve_8,~ x_3=y_3,~ |x-z| = |y-z|\}.
\end{multline*}

\medskip

\noindent We notice that the canonical relation $\Llg_j$, {\bf $j=1,\dots,4$} is defined by the {\bf rotations around the vertex} $\vv_j$ for all the element $$(y,\eta) \in \mV_{\vv_j}.$$
On the other hand, the canonical relation $\Llg_j$, {\bf $j=5,\dots, 8$} is defined by the {\bf rotations around the edge} $\ve_j$ (see Figure \ref{fig:edge}) of $\Ga$ for all $$(y,\eta) \in \overline \mV_{\ve_j}.$$

\begin{figure}
\begin{center}
    \includegraphics[width=0.3\textwidth]{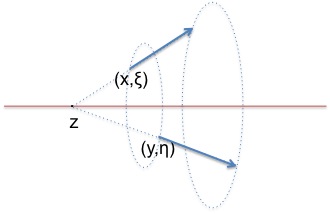}
 \end{center}
    \caption{Rotation around the edge (in red color).}
    \label{fig:edge}
\end{figure}

\medskip

\noindent The following result describes the wave front set of the Schwartz kernel $\mu$ of $\mT$

\begin{prop} \label{P:wavefront-3} We have
$$\wf(\mu)' \subset \Delta_\mV \, \bigcup \, \big(\bigcup_{j=1}^8 \Llg_j\big).$$
\end{prop}

\begin{proof}
The proof is similar to that of Proposition \ref{P:wavefront-3}. The reader is referred to \cite{FQ14} for detailed argument. We only sketch here the main idea. 

\medskip

We first notice that we can consider $\chi$ as a function of $(z'=(z_2,z_3),r,x) \in \rN^2 \times \rN_+ \times \Og$. Then its wave front set is described by $\chi$ 
$$\wf(\chi) \subset \big(\bigcup_{j=1}^4 \mA_j\big) \bigcup  \big(\bigcup_{j=5}^8 \mA_j \big).$$
Here,
\begin{eqnarray*}
\mA_j &=&\{(z',r,\eta,0;\, x,{\bf 0}) \in \cT^*(\rN^2 \times \rN_+) \times \cT^* \Og: z=\vv_j,~ \eta =(\eta_2, \eta_3) \neq 0\},\quad j=1,\dots,4, \\[6 pt]
\mA_j &=& \{(z',r, \eta,0; \, x,{\bf 0} ) \in \cT^*(\rN^2 \times \rN_+) \times \cT^* \Og: z \in \ve_j,~{\bf 0} \neq (0,\eta) \perp \ve_j\}, \quad j=5,\dots, 8.
\end{eqnarray*}
The rest follows from the calculus of wave front set of product and composition (see Appendix~\ref{A:Cal}). 
\end{proof}

\begin{rem} \label{R:micro-localize} Assume that $$\big(x,\xi= \tau(x-z_*); y, \eta= \tau(y-z_*) \big) \in \Llg_j,$$ for some $j=5,\dots,8$ and $\chi(z)$ vanishes around the point $z=z_*$. We observe that the following elements in $\Llg_j$ $$\{(z'_*,r, \eta,0; \, x,{\bf 0} ) \in \cT^*(\rN^2 \times \rN_+) \times \cT^* \Og: ~{\bf 0} \neq (0,\eta) \perp \ve_j\}$$ are excluded from $\wf(\chi)$. Applying the same argument as in the proof of Proposition \ref{P:wavefront-3}, we obtain 
$$(x,\xi= \tau(x-z_*); y, \eta= \tau(y-z_*)) \not \in \wf(\mu)'.$$
This observation provides the micro-localization argument used  later in the proof of Theorem~\ref{T:Main2}. 
\end{rem}

\noindent {\bf Discussion 3.} Let us describe some implications of Proposition \ref{P:wavefront-3}. We denote by $\pi_R$ the right projection on the product space $\cT^* \Og \times \cT^* \Og$. Then,
$$\pi_R(\Delta_\mV)= \mV, \quad  \pi_R(\bigcup_{j=1}^4 \Llg_j)= \mV_c, \quad  \pi_R(\bigcup_{j=1}^4 \Llg_j)= \overline \mV_e.$$
We also recall the following rule for wave front sets (see Appendix~\ref{A:Cal}):
$$\wf(\mT f) \subset \wf(\mu)' \circ \wf(f).$$
The following arguments follow from Proposition \ref{P:wavefront-3} \footnote{The element $(x,\xi)$ in the below discussion may play the role of $(y,\eta)$ in the definition of $\Llg_j$.}: 
\begin{itemize}
\item[a)] If $(x,\xi) \in \wf(f)$ is an invisible singularity, then $(x,\xi) \not \in \pi_R(\wf(\mu)')$. Therefore, $(x,\xi)$ is completely smoothened out by $\mT$.

\item[b)] If $(x,\xi) \in \wf(f)$ is a visible singularity, then $(x,\xi) \in \pi_R(\Delta_\mV)$ and $(x,\xi) \not \in \bigcup_{j=1}^8 \Llg_j$. Therefore, $(x,\xi)$ may be reconstructed but it does not generate any artifact.

\item[c)] If $(x,\xi) \in \wf(f)$ is a corner singularity, then $(x,\xi) \in \pi_R(\Delta_\mV)$ and $(x,\xi) \in \pi_R(\Llg_j)$ for one index $j=1,\cdots,4$. Therefore, $(x,\xi)$ may be reconstructed, and it may also generate artifacts by rotating around the vertex $\vv_j$.

\item[d)] If $(x,\xi) \in \wf(f)$ is an edge singularity, then $(x,\xi) \in \pi_R(\Delta_\mV)$ and $(x,\xi) \in \pi_R(\Llg_j)$ for one index $j=5,\cdots,8$. Therefore, $(x,\xi)$ may be reconstructed and it may generate the artifacts by rotating around the edge $\ve_j$.

\end{itemize}
The above observation is similar (although presented in a slightly different way) to \cite[Remark 4.7]{FQ14}. The following theorem will provide quantitative improvement for b), c), and d):
\begin{theorem}  \label{T:Main2} Let $\chi$ be as in (\ref{E:chi}). The following statements hold
\begin{itemize}
\item[a)]Microlocally on $\Delta \setminus (\bigcup_{j=1}^8 \Llg_j)$, we have $\mu \in I^0(\Delta)$ with the principal symbol 
$$\sg_0(x,\xi) = \chi(z),$$ where $z$ is the intersection of the line $\ell(x,\xi)$ with $\mS$.
\item[b)] For $j=1,\dots, 4$, microlocally on $\Llg_j \setminus \bigcup_{m =5}^8 \Llg_{m}$, $\mu \in I^{-2 \, k-1}(\Llg_j)$. 
\item[c)] For $j=5,\dots, 8$, microlocally on $\Llg_j \setminus \bigcup_{m =1}^4 \Llg_{m}$, $\mu \in I^{-k-\frac{1}{2}}(\Llg_j).$  
\end{itemize}
\end{theorem}
The reader is referred to Appendix~\ref{A:FIO} for basic facts about the space $I^m(\Llg_j)$. It is worth mentioning that, as shown in the proof below, the orders of $\mu$ on $\Llg_j$ stated in Theorem~\ref{T:Main2}~b)~\&~c) are optimal. We now proceed to prove Theorem~\ref{T:Main2}. 

\begin{proof}[{\bf Proof of Theorem~\ref{T:Main2}}] 
We will start our analysis with the following formula (see \cite{AMP})
\begin{equation} \label{E:mu-3} \mu(x,y) = \frac{x_1}{\pi^3}\intl_{\rN} \intl_{\Ga}e^{i(|x-z|^2-|y-z|^2) \, \llg}\,\llg^2 \, \chi(z) \, dz \, d\,\llg.\end{equation}

\medskip

\noindent{\bf Proof of a).} To prove a), we recall that the formula \eqref{E:mu-3} can be written as (see \cite{AMP})
\begin{eqnarray}\label{E:mu-33}
\mu(x,y) = \frac{1}{(2 \pi)^3} \intl_{\rN^3} e^{i \big(\left<x-y, \xi \right> - \, |x-y|^2 \frac{\xi_1}{2 \, x_1} \big)} \chi(z) \, d \xi. 
\end{eqnarray}

\noindent Here, as always, $z$ is the intersection of $\ell(x,\xi)$ and $\mS$. We note that $\sg_0(x,\xi) :=\chi(z)$ is smooth microlocally near any $(x,\xi)$ such that $(x,\xi;x,\xi) \in \Delta \setminus \bigcup_{j=1}^8 \Llg_j$. Moreover, it is homogenous of degree $0$ in $\xi$. Therefore, near $\Delta \setminus \bigcup_{j=1}^8 \Llg_j$,  $\mu \in I^{0}(\Delta)$. Moroever, the principal symbol of $\mu$ is $\sg_0(x,\xi)= \chi(z)$ (see, e.g., \cite[Theorem 3.2.1]{Soggeb}). 

\medskip

\noindent{\bf Proof of b).} Consider $j=1,\dots,4$ and $$(x^*, \xi;\, y^*, \eta) \in \Llg_j \setminus \bigcup_{m=5}^8 \Llg_m.$$ 
We will analyze $\mu$ microlocally near $(x^*,\xi; \, y^*,\eta)$.

\medskip

\noindent Let us rewrite \eqref{E:mu-3} as 
\begin{eqnarray*}  \mu(x,y) = \frac{x_1}{\pi^3} \intl_{\rN} \intl_{\Ga}e^{i\left<x-y,x+y-2z\right> \llg}\,\llg^2 \,  \chi(z) \, dz \, d\,\llg.\end{eqnarray*}
Therefore,
\begin{eqnarray*} \mu(x,y) = \frac{x_1}{\pi^3} \intl_{\rN} e^{i\left<x-y,x+y \right> \llg} \,\llg^2 \, \intl_{\Ga}e^{- i\left<x-y,2z\right> \llg}\, \chi(z) \, dz \, d\,\llg. \\
\end{eqnarray*}
That is, using the formula (\ref{E:chi}) of $\chi$, 
\begin{multline}\label{E:mu3-bisbis} \mu(x,y) = \frac{x_1}{\pi^3} \intl_{\rN} e^{i\left<x-y,x+y \right> \llg}\,\llg^2 \, \Big( \intl_{-a}^a e^{- 2 i(x_2-y_2) z_2 \llg}\, h_2(z_2) \, dz_2 \Big) \\  \times  \Big(\intl_{-b}^b e^{- 2 i(x_3-y_3) \, z_3  \llg} \,h_3(z_3) \, dz_3 \Big) \, d\,\llg.\end{multline}

\medskip

\noindent On the set $x_2 \neq y_2$, due to Lemma~\ref{L:sym},
\begin{eqnarray} \label{E:x2noty2}
 \intl_{-a}^a e^{- 2 i(x_2-y_2) z_2 \llg}\, h_2(z_2) \, dz_2 =   e^{- 2 i(x_2-y_2) a \llg}\, a_{1,+}(x,y,\llg)  +  e^{- 2 i(x_2-y_2) ) (-a)  \llg}\,a_{1,-}(x,y,\llg),
 \end{eqnarray}
where $a_{1,\pm}$ are classical symbol of order $-k-1$, whose the leading terms are $\frac{-h^{(k)}(a)}{[2\, i \, (x_2-y_2) \, \llg]^{k+1}}$ and $\frac{h^{(k)}(-a)}{[2\, i \, (x_2-y_2) \, \llg]^{k+1}}$, respectively.

\medskip

\noindent Similarly, on the set $x_3 \neq y_3$, we obtain
 \begin{eqnarray}\label{E:x3noty3} 
\intl_{-b}^b e^{- 2 i(x_3-y_3) \, z_3  \llg} \,h_3(z_3) \, dz_3= e^{- 2 i(x_3-y_3) b \llg} a_{2,+}(x,y,\llg) +  \, e^{- 2 i(x_3-y_3) ) (-b)  \llg}\, a_{2,-}(x,y,\llg), \end{eqnarray}
where $a_{2,\pm}$ are classical symbol of order $-k-1$, whose leading terms are $\frac{-h^{(k)}(b)}{[2\, i \, (x_3-y_3) \, \llg]^{k+1}}$ and  $\frac{h^{(k)}(-b)}{[2\, i \, (x_3-y_3) \, \llg]^{k+1}}$, respectively. 
\medskip

It is easy to observe that in a neighborhood of $(x^*,y^*)$, we have $x_2 \neq y_2$ and $x_3 \neq y_3$. Therefore, near $(x^*,y^*)$, due to \eqref{E:mu3-bisbis}, \eqref{E:x2noty2}, and \eqref{E:x3noty3}
\begin{equation*} \mu(x,y)  = \sum_{\eg_1 = \pm 1, \eg_2 = \pm 1} \intl_{\rN} e^{i\big[\left<x-y,x+y \right> - 2 (x_2 - y_2) (\eg_1 a) + (x_3-y_3) ( \eg_2 b)\big] \llg} a_{\eg_1, \eg_2}(x,y,\llg) \,  d\,\llg,
\end{equation*}
where $a_{\eg_1,\eg_2}$ is a classical symbol of order $-2k$, with nonzero leading term. That is,
\begin{eqnarray*} \mu(x,y)  &=& \sum_{m=1}^4 \mu_m.
\end{eqnarray*}
Here, for $m=1,\dots,4$:
\begin{eqnarray}\label{E:muk} \mu_m(x,y)  = \intl_{\rN} e^{i (|x-\vv_m|^2 - |y-\vv_m|^2) \llg} \, a_m(x,y, \llg)   d\,\llg,\end{eqnarray} where $a_m$ is a classical symbol of order $-2k$, with nonzero leading term \footnote{We can easily write down the leading term of $a_m$ in terms of $h_2^{(k)} \otimes h_3^{(k)}(\vv_m')$. However, such formula is not essential for our argument}. 
\medskip

From the standard theory of FIO (see also Appendix~\ref{A:FIO}), we obtain near $(x^*,y^*)$
$$\wf(\mu_m) \subset \Llg_m, \quad \mbox{ for all } m=1,\dots, 4.$$
Due to the assumption $(x^*, \xi; y^*, \eta) \in \Llg_j$, $(x^*, \xi; y^*, \eta) \not \in \wf(\mu_m)$ for $m \neq j$. Therefore, microlocally near $(x^*,\xi; y^*, \eta)$, $$\mu = \mu_j  \mbox{ modulo a smooth function}.$$  Due to \eqref{E:muk} (see also Appendix~\ref{A:FIO}), $\mu_j \in I^{-2k-1}(\Llg_j)$ microlocally near $(x^*,\xi; y^*, \eta)$, and so is $\mu$. This finishes the proof for b). We notice that, since the leading term of $a_j$ is nonzero, the order $-2k-1$ of $\mu$ on $\Llg_j$ is optimal. 

\medskip

\noindent {\bf Proof of c).} We now prove c) for $j=5,6$. The case $j=7,8$ is similar. Let us consider  $$(x^*,\xi; y^*, \eta) \in \Llg_j \setminus \bigcup_{k = 1}^4 \Llg_{m}.$$ Then, there is $z^*=(0,-a<s_0<a,\pm b)$ such that
$$\xi = \llg \, (x^* - z^*),~\eta= \llg \, (y^* - z^*).$$
Let us pick a function $g \in C^\infty_0(-a,a)$ such that $g(s) = h_2(s)$ around the point $s=s_0$. We then can write
$$\chi(z) = g(z_2) h_3(z_3) + (h_2(z_2)-g(z_2)) h_3(z_3) := \chi_1(z) + \chi_2(z).$$
Therefore, 
$$\mu = \mu_1 + \mu_2,$$ where $\mu_1,\mu_2$ are respectively the Schwartz kernel of 
$$\mT_1=  \frac{x_1}{\pi^3}\mR^*\mP  \chi_1  \mR \mbox{ and } \mT_2 := \frac{x_1}{\pi^3}\mR^* \mP  \chi_2 \mR.$$
Using Remark~\ref{R:micro-localize}, we obtain (since $\chi_2$ vanishes near $z^*$) $$(x^*,\xi; y^*, \eta) \not \in \wf(\mu_2)'.$$
Therefore, in order to analyze $\mu$ microlocally at $(x^*,\xi; y^*, \eta)$, it suffices to analyze $\mu_1$. 
Similarly to \eqref{E:mu3-bisbis}, we can write
\begin{eqnarray*} \mu_1(x,y) &=& \frac{x_1}{\pi^3} \intl_{\rN} e^{i\left<x-y,x+y \right> \llg} \,\llg^2 \, \intl_{\Ga}e^{- i\left<x-y,2z\right> \llg}\, \chi_1(z) \, dz \, d\,\llg \\ &=&  \frac{x_1}{\pi^3} \intl_{\rN} e^{i\left<x-y,x+y \right> \llg}\,\llg^2 \, \Big( \intl_{-a}^a e^{- 2 i(x_2-y_2) z_2 \llg}\, g(z_2) \, dz_2 \Big) \, \Big(\intl_{-b}^b e^{- 2 i(x_3-y_3) \, z_3  \llg} \,h_3(z_3) \, dz_3 \Big) \, d\,\llg.\end{eqnarray*}
We observe that in a neighborhood of $(x^*, y^*)$, we have $x_3 \neq y_3$. Using \eqref{E:x3noty3}, we obtain locally near $(x^*,y^*)$ 
$$\mu_1 = \gamma_+ + \gamma_-,$$
where
\begin{eqnarray*} \gamma_\pm (x,y) &=&  \intl_{\rN} e^{i\big(\left<x-y,x+y  \right> -2(x_3 - y_3)(\pm b)  \big) \llg} \,\Big( \intl_{-a}^a e^{- 2 \, i \, (x_2-y_2)\, s \, \llg}\, g(s) \, ds \Big) \, a_\pm(x,y,\llg) \,  d\,\llg ,
\end{eqnarray*} and $a_\pm$ is a classical symbol of order $1-k$, with nonzero leading term. 
Let us denote $$x''=(x_1,x_3), \quad y''=(y_1,y_3), \quad z''_{\pm}=(0,\pm b).$$ Then,
\begin{eqnarray*} \gamma_\pm (x,y) &=&  \intl_{\rN} e^{i\big(\left<x''-y'',x''+y''-2z''_\pm \right>+(x_2^2-y_2^2) \big) \llg} \,\Big( \intl_{-a}^a e^{- 2 \, i \, (x_2-y_2)\, s \, \llg}\, g(s) \, ds \Big) \, a_\pm (x,y,\llg) \,  d\,\llg.
\end{eqnarray*}
That is, 
\begin{eqnarray*} \gamma_\pm(x,y) &=& \intl_{-a}^a  \intl_{\rN} e^{i \big[\left<x''-y'',x''+y''-2z''_\pm\right> \, \llg +(x_2 -y_2) (x_2+y_2 - 2 s) \, \llg \big] }  \, a_\pm(x,y,\llg) \,   g(s)\, d\,\llg   \, ds\\ &=& \intl_{-a}^a  \intl_{\rN} e^{i \big[( |x''-z''_\pm|^2 - |y''-z''_\pm|^2)\, \llg +2 (x_2 -y_2) (x_2-  s) \, \llg - (x_2-y_2)^2 \, \llg \big] }  \, a_\pm(x,y,\llg) \,   g(s) \,  d\,\llg   \, ds .\end{eqnarray*}
Let us consider the change of variables $$(\llg, s) \to \xi = (\llg, (x_2-s) \llg),$$ whose Jacobian is $$J(\llg, s) = - \llg.$$
We then arrive to
\begin{eqnarray*}\gamma_\pm(x,y) =  \intl_{\rN^2} e^{i \big[( |x''-z''_\pm|^2 - |y''-z''_\pm|^2)\, \xi_1 +2 (x_2 -y_2) \xi_2 - (x_2-y_2)^2 \, \xi_1 \big] }  \, \frac{a_\pm(x,y,\xi_1)}{|\xi_1|} \,   \sg(x,\xi) \, d\xi,\end{eqnarray*}
where $$\sg(x,\xi) = g\big(x_2 -\frac{\xi_2}{\xi_1})$$ is homogenous of order zero in $\xi$.

\medskip

\noindent Since $\supp g \subset (-b,b)$, $\sg(x,\xi)\equiv 0$ in a neighborhood of $(x^*,y^*)$ if $|\xi_2| \geq C|\xi_1|$ for some fixed constant $C$. Therefore, 
$$\frac{a_\pm(x,y,\xi_1)}{|\xi_1|} \,   \sg(x,\xi) \mbox{ is a classical symbol of order $-k$ around }(x^*,y^*).$$ 
Applying the standard theory of FIO (see also Appendix~\ref{A:FIO}), we obtain that in a neighborhood of $(x^*,y^*)$, $$\gamma_1 \in I^{-k-1/2}(\Llg_5) \mbox{ and } \gamma_2 \in I^{-k-1/2}(\Llg_6). $$

\medskip

\noindent Therefore, $\mu_1 = \gamma_1 + \gamma_2 \in I^{-k-\frac{1}{2}}(\Llg_j)$ microlocally near $(x^*,\xi; y^*, \eta)$, and hence so is $\mu$. This finishes the proof of the Theorem \ref{T:Main2}~c). We finally notice that, since the leading term of $a_\pm$ is nonzero, the order $-k-\frac{1}{2}$ of $\mu$ on $\Llg_j$ is optimal.
\end{proof}


\noindent {\bf Discussion 4.} The following improvement  of statements b), c), and d) in Discussion 3 are clear consequences of Theorem \ref{T:Main2} (and Lemmas~\ref{L:x0},~\ref{L:ma}) \footnote{As in Discussion 3, the element $(x,\xi)$ in the below discussion may play the role of $(y,\eta)$ in the definition of $\Llg_j$.}:

\begin{itemize}

\item[b')] Let $(x,\xi) \in \wf(f)$ be a visible singularity. Then, due to Theorem~\ref{T:Main2}~a), $(x,\xi)$ is reconstructed by $\mT$ as long as $\sg_0(x,\xi) =\chi(z) \neq 0$. Moreover, the reconstructed singularity is of the same order as the original singularity.

\item[c')] Let $(x,\xi) \in \wf(f) \cap \mV_{\vv_j}$ be a corner singularity. Due to Theorem~\ref{T:Main2}~b) and Lemma~\ref{L:x0}, an artifact at $(y,\eta) \neq (x,\xi)$, generated by $(x,\xi)$, is $2k$ orders smoother than the original singularity, if  $(y,\eta)$ cannot be obtained from $(x,\xi)$ by a rotation around an edge.

\item[d')] Let $(x,\xi) \in \wf(f) \cap \mV_{\ve_j}$ be an edge singularity. Due to Theorem~\ref{T:Main2}~c) and Lemma~\ref{L:ma}, the artifacts it generates are $k$ orders smoother than itself. 

\end{itemize}
We, however, still cannot describe the reconstructions of the boundary singularities (as in two dimensional problem). Moreover, it is not clear how strong the artifact is if it is obtained from an original singularity by both a rotation around an edge and a vertex. These issues are left to a future research.

\section{Concluding remarks} \label{S:Re}
It is clear that our results can be generalized to any $n$ dimensional space, where $\Gamma$ is a generalized rectangle in $\rN^{n-1}$. 

\medskip

\noindent We point out here several possible research directions that may grow out from this article
\begin{itemize}
\item[1)] In this article, we assume that  the observation surface $\Ga$ is flat. More precisely, we assume $$\Ga=\{0\} \times  [-1,1],\quad  \mbox{ for } n=2,$$ and $$\Ga=\{0\} \times [-a,a] \times [-b,b], \quad \mbox{ for } n =3.$$ 
It is natural to ask what happens if $\Ga$ has other shapes? This issue will be addressed in our up coming work \cite{Artifact-Sphere-Curve}.

\item[2)] Another direction of study is the investigate the full data problem when the observation surface $\mS$ is not smooth (although enclosing the support of $f$). This has practical application, since the rectangular geometry has been common used in practice.

\item[3)] The method developed in this article can be used to study the limited data problems for other integral transforms, such that elliptical and cylindrical transforms. 
\end{itemize}

Finally, it is also worth mentioning again that questions raised in Discussions~2 and 4 are also interesting issues to investigate. 

\section*{Acknowledgement}
The author is grateful to Professor T. Quinto and Dr J. Frikel for kindly communicating their results in \cite{FQ14}, which helps to speed up this work. He is also thankful to Professor G. Uhlmann for introducing him to the theory of pseudo-differential operators with singular symbols, whose spirit inspires this and the previous work \cite{Streak-Artifacts}. 

\appendix
\section{Appendix}
The appendix consists of three parts. Appendix~\ref{A:Asym} is dedicated to some asymptotics results needed for the proof of Theorems~\ref{T:Main1} and \ref{T:Main2}. In Appendix~\ref{A:Cal}, we state some basic rules for calculus of wave front set which are used through out the article. Finally, two types of Fourier distributions are introduced in Appendix~\ref{A:FIO}. The first type has the canonical relation defined by the rotations around a point, whereas the second type defined by rotations around a line. 
\subsection{Asymptotics of an oscillatory integral} \label{A:Asym}
The following asymptotic results are used in the proof of Theorems~\ref{T:Main1} and \ref{T:Main2}.
\begin{lemma}\label{L:sym}  Let $h \in C^\infty([c,d])$, we define
\begin{eqnarray*} A (s,t,\llg) = \intl_{c}^d e^{- 2 \, i \, (s-t) \, \tau\, \llg}\, h(\tau) \, d\tau.\end{eqnarray*}
Then, $A \in C^\infty(\rN \times \rN \times \rN)$. Moreover, 
\begin{itemize}
\item[i)] If $\supp(h) \subset [c,d)$ and $h$ vanishes to order $k$ at $\tau =c$, then
\begin{eqnarray*} A (s,t,\llg) = e^{ - 2 \, i \, (x_2-y_2) \,c \, \llg} \, a(s,t,\llg),
\end{eqnarray*}
where $a(s,t,\llg)$ is a classical symbol of order $-k-1$ when $s \neq t$, with the leading term $$a_{-k-1}(s,t,\llg) = \frac{h^{(k)}(c)}{[2\, i \,(s-t) \llg]^{(k +1)}}.$$

\item[ii)] If $\supp(h) \subset (c,d]$ and $h$ vanishes to order $k$ at $\tau =d$, then
\begin{eqnarray*} A (s,t,\llg) = e^{ - 2 \, i \, (s-t) \,d \, \llg} \, a(s,t,\llg),
\end{eqnarray*}
where $a(s,t,\llg)$ is a classical symbol of order $-k-1$ when $s \neq t$, with the leading term $$a_{-k-1}(s,t,\llg) = \frac{- h^{(k)}(d)}{[2\, i \,(s-t) \llg]^{(k +1)}}.$$
\item[iii)] If $h$ vanishes to order $k$ at $\tau=c,d$, then
\begin{eqnarray*}A (s,t,\llg) = e^{ - 2 \, i \, (x_2-y_2) \,c \, \llg} a_c(s,t,\llg) + e^{ - 2 \, i \, (s-t) \,d \, \llg} \, a_d(s,t,\llg),
\end{eqnarray*}
where $a_c$ and $a_d$ are classical symbol of order $-k-1$ when $s \neq t$, with the leading terms respectively
$$a_{c,-k-1}(s,t,\llg) = \frac{h^{(k)}(c)}{[2\, i \,(s-t) \llg]^{(k +1)}},\quad a_{d,-k-1}(x,y,\llg) = \frac{- h^{(k)}(d)}{[2\, i \,(s-t) \llg]^{(k +1)}}.$$
\end{itemize}
\end{lemma}

\begin{proof} The smoothness of $A$ is obvious. It remains to prove i), ii), and iii).

\medskip

We only need to prove i). The proof of ii) is similar. The proof of iii) can be reduced to those of i) and ii) by decomposing $h = h_c + h_d$, where $h_c,h_d$ satisfy the condition of $h$ in i) and ii), respectively.

\medskip

Let us proceed to prove i). We have, using integration by parts, for $s \neq t$
\begin{eqnarray*} A (s,t,\llg) &=& \frac{1}{-2 i \,(s-t) \llg }  \intl_{c}^d \big[e^{- 2 \, i \, (s-t) \, \tau\, \llg}\big]_\tau \, \, h(\tau) \, d\tau
\\ &=& \frac{1}{-2 i \,(s-t) \llg } e^{- 2 \, i \, (s-t) \, \tau\, \llg}  \, h(\tau)\Big|_{c}^{d} \\ &-& \frac{1}{-2 i \,(s-t) \llg } \intl_{c}^d e^{- 2 \, i \, (s-t) \, \tau\, \llg} \, h'(\tau) \, d\tau
.\end{eqnarray*}
Since $h(\tau) = 0$ near $\tau =d$, we obtain
\begin{eqnarray*} A(s,t,\llg) &=& \frac{h(c)}{2 i \,(s-t) \llg } e^{- 2 \, i \, (s-t) \,c\, \llg} \\ &+& \frac{1}{2 i \,(s-t) \llg } \intl_{c}^d e^{- 2 \, i \, (s-t) \, \tau\, \llg} \, h'(\tau) \, d\tau
.\end{eqnarray*}
Continuing the integration parts, we arrive to
\begin{eqnarray*} A (s,t,\llg) &=& \sum_{l=0}^{k+1} \frac{h^{(l)}(c)}{[2 i \,(s-t) \llg]^{l+1} } e^{- 2 \, i \, (s-t) \,c\, \llg} \\ &+& \frac{1}{[2 i \,(s-t) \llg]^{k+2} } \intl_{c}^d e^{- 2 \, i \, (s-t) \, \tau\, \llg} \, h^{(k+2)}(\tau) \, d\tau
.\end{eqnarray*}
Since, $h^{(l)}(c) =0$ for $0 \leq l \leq k-1$,
\begin{eqnarray*} A(s,t,\llg)  &=&e^{- 2 \, i \, (s-t) \,c\, \llg} \, \Big\{\frac{h^{(k)}(c)}{[2 i \,(s-t) \llg]^{k+1} }  + \frac{h^{(k+1)}(c)}{[2 i \,(s-t) \llg]^{k+2} }  \\ &+& \frac{1}{[2 i \,(s-t) \llg]^{k+2} } \intl_{c}^d e^{- 2 \, i \, (s-t) \, (\tau-c) \, \llg} \, h^{(k+2)}_c(\tau) \, d\tau\Big\}.\end{eqnarray*}
Using similar integration by parts, we can easily prove that $$\intl_{c}^d e^{- 2 \, i \, (s-t) \, (\tau-c) \, \llg} \, h^{(k+2)}(\tau) \, d\tau$$ is a classical symbol of order $0$ when $s \neq t$. This finishes our proof.

\end{proof}

\subsection{Calculus for wave front sets} \label{A:Cal}
We extract here some rules for calculus of wave front sets (see, e.g., \cite{hormander71fourier,Ho1}):
\begin{itemize}
\item[1)] {\bf Propagation of wave front set under linear map:} Let $\mT$ be the linear transformation defined by the Schwartz kernel $\mu \in \mD'(X \times Y)$ satisfying $\wf(\mu)' \subset (\cT^* X \setminus 0) \times (\cT^* Y \setminus 0)$. Then
$$\wf(\mT u) \subset \wf(\mu)' \circ \wf(u).$$ 
This rule is used in Discussions~1,2,3 and 4 to explain the reconstruction of singularities (due to the wavefront set of $\mu$ on $\Delta$) and generation of singularities (due the the wavefront set of $\mu$ on other Lagrangian manifolds).

\medskip

\item[2)] {\bf Product rule:} Let $u,v$ be two distributions on the same space $X$. Then the product $uv$ is well defined unless $(x,\xi) \in \wf(u)$ and $(x,-\xi) \in \wf(v)$ for some $(x,\xi)$. Moreover,
$$\wf(uv) \subset \{(x,\xi+\eta): (x,\xi) \in \wf(u) \mbox{ or } \xi=0, (x,\eta) \in \wf(v) \mbox{ or } \eta =0\}. $$
This rule is used in the proof of Propositions \ref{P:wavefront} and \ref{P:wavefront-3}. 

\medskip

\item[3)] {\bf Composition rule:} Let $\mT_1, \mT_2$ be linear transformations defined by the Schwartz kernels $\mu_1 \in \mD'(X \times Y)$, $\mu_2 \in \mD'(Y \times Z)$ satisfying $\wf(\mu_1) \subset (\cT^* X \setminus 0) \times (\cT^* Y \setminus 0)$ and $\wf(\mu_2) \subset (\cT^* Y \setminus 0) \times (\cT^* Z \setminus 0)$. Let $\mu$ be the Schwartz kernel of $\mT_1 \circ \mT_2$, then
$$\wf(\mu)' \subset \wf(\mu_1)' \circ \wf(\mu_2)' .$$
In particular, if $\wf(\mu_1)=-\wf(\mu_1)$ and $\wf(\mu_2) = -\wf(\mu_2)$, then
$$\wf(\mu)' \subset \wf(\mu_1) \circ \wf(\mu_2).$$
This rule is used in the proof of Propositions \ref{P:wavefront} and \ref{P:wavefront-3}. 
\end{itemize}

\subsection{Related Fourier Distributions} \label{A:FIO}
In this section, we only introduce some special Fourier distributions which are needed in this article. The reader is referred to, e.g., \cite{hormander71fourier,TrFour} for the general theory of the topic. Let $\Og \subset \rN^n$ be an open set and $\phi=\phi(x,y,\llg) \in C^\infty(\Og \times \Og \times (\rN^N \setminus 0))$ satisfy \begin{itemize} \item[1)] $\phi$ is homogeneous of degree $1$ in $\llg$, \item[2)] $\phi_x \neq 0$ and $\phi_y\neq 0$ on the set $$\mC=\{(x,y,\llg) \in \Og \times \Og \times (\rN^N \setminus 0): d_\llg\phi =0\}.$$
\end{itemize}
Such as function $\phi$ is called an {\bf operator phase function}. Let us define $\Llg \subset (\cT^* \Og \setminus 0) \times (\cT^* \Og \setminus 0)$ by \begin{equation*} \label{De:can} \Llg= \{(x,d_x \phi; y, - d_y \phi): (x,y,\llg) \in \mC \}. \end{equation*}
Then, $\Llg$ is called the homogeneous {\bf canonical relation} associated to $\phi$.

\medskip

Let us define $\mu \in \mD'(\Og \times \Og)$ by
$$\mu(x,y) = \intl_{\rN^N} e^{i \phi(x,y,\llg)} \, a(x,y,\llg) \,d\llg,$$
where $a(x,y,\llg) \in S^{m+(n-N)/2} (\Og \times \Og \times \rN^N)$ \footnote{The reader is referred to \cite{hormander71fourier,TrFour} for the definition of the symbol class $S^p(\Og \times \Og \times \rN^N)$. The order $p=m+(n-N)/2$, required in the above definition, is specified in, e.g., \cite[pp. 115]{hormander71fourier}.}. Then, $\mu$ is a Fourier distribution of order $m$ whose canonical relation is $\Llg$. We write $\mu \in I^{m}(\Llg)$. The linear operator $\mT: \mD(\Og) \to \mD'(\Og)$ whose Schwartz kernel is $\mu$ is a Fourier integral operator (FIO) of order $m$. With a slight abuse of notation, we also write $\mT \in I^{m}(\Llg)$.

\medskip

The following result (see \cite[Theorem 4.3.2]{hormander71fourier}) will be used to analyze the mapping properties of the FIOs discussed below: 

\begin{theorem} \label{T:Ho}
Let $\Llg \subset (\cT^* \Og \setminus 0) \times (\cT^* \Og \setminus 0)$ be a homogeneous canonical relation such that both of its left and right projections on $\Og$ have surjective differentials. Assume that the differentials of the left and right projections $\Llg \to \cT^* \Og$ have rank at least $k + n$. Then, every $\mT \in I^m(\Llg)$ maps continuously from $H^s_{comp}(\Og)$ to $H_{loc}^{s- m - \frac{n-k}{2}}(\Og)$. 
\end{theorem}

\medskip

In the below discussion, we introduce two classes of Fourier distributions whose canonical is defined by rotations around a point or a line, respectively.

\medskip

\noindent{\bf Fourier distributions associated to rotations around a point.} Let us now introduce the class of Fourier distributions whose canonical relation is defined by the rotations around a point. This class of Fourier distributions is used in the statement and proof of Theorem~\ref{T:Main1}~b) (namely, the class $I^m(\Llg_\pm)$) and Theorem~\ref{T:Main2}~b) (the class $I^m(\Llg_j)$, $j=1,\dots,4$). 

\medskip

Let $x_0 \in \rN^n$ such that $x_0 \not \in \Og$. We define
\begin{equation} \label{E:FIO} \mu_{x_0}(x,y) =  \intl_{\rN} e^{i (|x-x_0|^2-|y-x_0|^2) \, \llg} a(x,y,\llg)  d\llg,\end{equation} where $a \in S^{m+\frac{n-1}{2}}(\Og \times \Og \times \rN)$. 
Then $\mu_{x_0} \in I^m(\Llg_{x_0})$ where  $$\Llg_{x_0}=\{(x,\tau (x-x_0);y, \tau(y-y_0)) \in (\cT^* \Og \setminus 0) \times (\cT^* \Og \setminus 0): |x-x_0| = |y-x_0| \} $$ is the canonical relation defined by the rotations around $x_0$. 

\medskip

The following result, which is a direct consequence of Theorem~\ref{T:Ho}, is used in Discussions~2~\&~4:
\begin{lemma} \label{L:x0}
Let $\mT$ be the linear operator whose Schwartz is $\mu_{x_0}$. Then, $\mT$ maps continuously from $H^s_{comp}(\Og)$ to $H^{s-m-\frac{n-1}{2}}_{loc} (\Og)$.
\end{lemma}


\begin{proof}
We only need to apply Theorem~\ref{T:Ho} with $k=1$. 
\end{proof}

\medskip

\noindent{\bf Fourier distributions associated to rotations around a line.} Let us now only consider $\Og \subset \rN^3$. We introduce a class of Fourier distributions, whose canonical relation is defined by the rotations around a straight line. This class of Fourier distributions appears in the statement and proof of Theorem~\ref{T:Main2}~c) (the class $I^m(\Llg_j)$, $j=5,\dots, 8$). For notational ease, given each $x = (x_1,x_2,x_3) \in \rN^3$, we will write $x'=(x_1,x_2)$ and $x''=(x_1,x_3)$. 

\medskip

Let $\ma=(a_1,a_2) \in \rN^2$ such that the vertical line 
 $$\ell_{v,\ma} :=\{z \in \rN^3: z'=\ma\}= \{\ma\} \times \rN$$
does not intersect $\Og$. We define the Fourier distribution $\mu_\ma \in \mD'(\Og \times \Og)$ by
\begin{eqnarray*}\mu_{\ma} (x,y) =  \intl_{\rN^2} e^{i \big[( |x'-\ma|^2 - |y'-\ma|^2)\, \xi_1 +2 (x_3 -y_3) \xi_2 - (x_3-y_3)^2 \, \xi_1 \big] }  \, a(x,y,\xi) \, d\xi,\end{eqnarray*}
where $a(x,y,\xi) \in S^{m+\frac{1}{2}}(\Og \times \Og \times \rN^2)$. Then $\mu_\ma \in I^{m}(\Llg_\ma)$, where
\begin{multline*}
\Llg_\ma = \{(x, \tau \, (x-z); \, y, \tau \, (y-z)) \in (\cT^*\Og \setminus 0) \times (\cT^*\Og \setminus 0): x_3=y_3, |x-z| = |y-z|, z\in \ell_{v,\ma}\}
\end{multline*} is the canonical relation defined by the rotations around the vertical line $\ell_{v,\ma}$.

\medskip

Similarly, let $\mb=(b_1,b_3) \in \rN^2$ such that the horizontal line  
 $$\ell_{h,\mb} :=\{z \in \rN^3: z''=\mb\}= \{b_1\} \times \rN \times \{b_3\}$$
does not intersect $\Og$. We define 
\begin{eqnarray*}\mu_{\mb} (x,y) =  \intl_{\rN^2} e^{i \big[( |x''-\mb|^2 - |y''-\mb|^2)\, \xi_1 +2 (x_2 -y_2) \xi_2 - (x_2-y_2)^2 \, \xi_1 \big] }  \, a(x,y,\xi) \, d\xi,\end{eqnarray*}
where $a(x,y,\xi) \in S^{m+\frac{1}{2}}(\Og \times \Og \times \rN^2)$. Then $\mu_\mb \in I^{m}(\Llg_\mb)$, where
\begin{multline*}
\Llg_\mb = \{(x, \tau \, (x-z); \, y, \tau \, (y-z)) \in (\cT^*\Og \setminus 0) \times (\cT^*\Og \setminus 0): x_2=y_2, |x-z| = |y-z|, z\in \ell_{h,\mb}\}
\end{multline*}
is the canonical relation defined by the rotations around the horizontal line $\ell_{h,\mb}$.

\medskip

The following lemma, which is a direct consequences of Theorem~\ref{T:Ho}, is used to analyze the strength of the edge artifacts in three dimensional problem (see Discussion~4):

\begin{lemma} \label{L:ma}
Let $\mT$ be the linear operator whose Schwartz is $\mu_{\ma}$ or $\mu_\mb$. Then, $\mT$ maps continuously from $H^s_{comp}(\Og)$ to $H^{s-m-\frac{1}{2}}_{loc} (\Og)$.
\end{lemma}
\begin{proof}
We only need to apply Theorem~\ref{T:Ho} with $n=3$ and $k=2$. 
\end{proof}

\newcommand{\etalchar}[1]{$^{#1}$}
\def\dbar{\leavevmode\hbox to 0pt{\hskip.2ex \accent"16\hss}d}

\end{document}